\documentclass{article}

\usepackage{arxiv}

\usepackage[utf8]{inputenc} 
\usepackage[T1]{fontenc}    
\usepackage{hyperref}       
\usepackage{url}            
\usepackage{booktabs}       
\usepackage{amsfonts}       
\usepackage{nicefrac}       
\usepackage{microtype}      
\usepackage{lipsum}
\usepackage{graphicx}
\graphicspath{ {./images/} }

\usepackage{cite}
\usepackage{amsmath,amssymb,amsfonts}
\usepackage{amsthm}
\usepackage{amsbsy}
\usepackage{algorithmic}
\usepackage{graphicx}
\usepackage{textcomp}
\usepackage{xcolor}
\def\BibTeX{{\rm B\kern-.05em{\sc i\kern-.025em b}\kern-.08em
    T\kern-.1667em\lower.7ex\hbox{E}\kern-.125emX}}
\usepackage{stmaryrd}
\usepackage{url}
\usepackage{enumerate}
\usepackage[english]{babel}

\newtheorem{theorem}{Theorem}  

\newtheorem{proposition}{Proposition}
\newtheorem{corollary}{Corollary}

\newtheorem{ex}{Example}

\def\AE{\mathrm{A}}
\def\APA{\mathrm{B}}
\def\F{\mathrm{C}}
\def\C{\mathrm{D}}

\definecolor{rouge}{rgb}{1,0,0}
\definecolor{bleu}{rgb}{0,0, 1}

\newcommand{\Omit}[1]{}

\def\Univ{\mathcal{C}}
\def\muc{\mu^c}

\usepackage{tikz}
\usetikzlibrary{decorations.text}
\def\dpr{\hbox{\scriptsize \rm DPR}}

\title{Qualitative Integrals on Dragonfly Algebras\thanks{The first author acknowledges the support of Czech Science Foundation
through the grant 20-07851S. The second author has been supported by the ERDF/ESF project AI-Met4AI
No. CZ.02.1.01/0.0/0.0/17\textunderscore049/0008414.}}

\author{
 Anton\'{\i}n Dvo\v{r}\'ak \\
  Institute for Research and\\
{Applications of Fuzzy Modeling} \\
{University of Ostrava}\\
 Czech Republic\\
  \texttt{antonin.dvorak@osu.cz} \\
   \And
Michal Hol\v{c}apek \\
  Institute for Research and\\
{Applications of Fuzzy Modeling} \\
{University of Ostrava}\\
 Czech Republic\\
  \texttt{michal.holcapek@osu.cz} \\
  \And
 Agn\`es Rico \\
{Entrepôt Représentation et}\\
{Ingénierie des Connaissances}\\
{Universit\'e de Lyon} \\
France\\
\texttt{agnes.rico@univ-lyon1.fr}

}

\begin{document}
\maketitle
\begin{abstract}
In this paper, we investigate qualitative integrals (generalizations of Sugeno integral) acting on recently introduced Dragonfly algebras. These algebras are designed for applications in data analysis (based on fuzzy relational compositions) when some data are unknown (e.g., missing). Unknown data are represented by the additional dummy value. Definitions of operations on Dragonfly algebras follow lower estimation strategy, i.e., results of operations can be interpreted as lower estimations of results obtained when unknown values are replaced by known ones. We define qualitative integrals on Dragonfly algebras and show their monotonicity. We also prove results characterizing when these integrals return known (or unknown) values. We also present an illustrative example and hint directions of further research.\end{abstract}


\section{Introduction}
In many knowledge representation models, the question of expressing ignorance arises. In \cite{TID03}, the unknown value belongs to the set $\{-,0,+\}$ and it is the sign of the result of the sum between real values. 
In multi-criteria decision making, data are collections of pairs $(f, \alpha)$, where $f$ is a vector of the evaluations of the object according to criteria and $\alpha$ represents a global evaluation given for this object by an expert. 
We can found pieces of data in which some coordinates of $f$ are unknown, but $\alpha$ is known (see, for example, \cite{DDPRF15}).

Three-valued logics deal with a third truth value different from the standard true and false ones. 
Three-valued logics have been used for different purposes depending on the meaning of the third value. 
In Kleene Logic \cite{K52}, the third value is interpreted as unknown.
Recently, a new approach was proposed \cite{SCBBD19} introducing so-called Dragonfly algebras. It adds a new element $*$ into the structure of truth values and interprets it as an unknown (e.g., missing) value. The primary field of application of Dragonfly algebras lies in fuzzy relational compositions, where it is important to follow a lower estimation approach in determining the resulting values of these compositions. The name ``Dragonfly'' refers to the successful real application  providing taxonomical classification of dragonflies.  

Considering $L$ a residuated lattice, and $*$ the unknown value, the Dragonfly algebra $L^* = L \cup \{*\}$ is a weaker algebraic structure than the original $L$. 
{In order to keep some important structural properties, e.g., associativity, we have to suppose that $L$ has no zero divisors.
With enough properties, $L^*$ can be used as an evaluation scale for objects or alternatives. In such a context, qualitative integrals are known to be suitable aggregation functions \cite{DH12,DPR16}. 
The qualitative integrals are defined using monotone set functions named capacities or fuzzy measures. 
The values of these set functions can happen to be unknown for some subsets of the universe.  Unfortunately, residuum-based qualitative integrals lose their monotonic properties in this case.
Nevertheless it is possible to characterize the cases for which these integrals return an unknown value.
Using an unknown value from a Dragonfly algebra we express the lower estimation of truth values we could guarantee for the result.} In some cases this approach can have some drawbacks. We will present this situation using data in which the expert gives a global evaluation for some  alternatives even if some local evaluations are unknown (see \cite{F} for more details). 
In this example, we suppose that the evaluation scale is a totally ordered set.
Motivated by the results obtained for this data, we outline possible approaches to be pursued in further research.

The paper is structured as follows. In Section~\ref{s.prel.}, we introduce residuated lattices and qualitative integrals on them. In the central Section~\ref{s.int_on_DA}, we first present the definition of Dragonfly algebras and their basic properties. Then we define multiplication-based and residuum-based qualitative integrals on Dragonfly algebras and show their properties. Section~\ref{s.example_unknown_values} contains an example with data representing local and global evaluations of objects, where some local evaluations are missing. We demonstrate the ability of qualitative integrals on Dragonfly algebras to aggregate these data with missing values. 

\section{Preliminaries}\label{s.prel.}
This section presents residuated lattices as the basic structure of truth values for modeling the values of input vectors and the multiplication and residuum based qualitative integrals of such input vectors. Later, the residuated lattices will be extended to the Dragonfly algebras.  
 
\subsection{The structure of truth values}

The structure of truth values is a linearly ordered residuated lattice, i.e.  the algebraic structure $(L, \wedge, \vee, \otimes, \to,  0, 1) $ such that 
 $(L, \wedge, \vee, 0,1)$ is a linearly ordered lattice with the bottom element $0$ and the top element $1$,
 $(L,\otimes, 1)$ is a commutative 
 monoid, and the adjointness property for $\otimes$ and $\to$
is satisfied, i.e.,  
\begin{gather} \label{eq.adjointness}
 \alpha \otimes \beta \leq \gamma \hbox{ iff } \alpha \leq \beta \to \gamma,
 \end{gather}
 where $\leq$ is the lattice ordering. The operations $\otimes$ and $\to$ are called the \emph{multiplication} and \emph{residuum}. Note that for any $\alpha,\beta \in L$, it holds that $\alpha\leq \beta$ if and only if $ \alpha\to \beta =1$.  
 Define $\neg \alpha=\alpha\to 0$, the negation of $\alpha$ for any $\alpha\in L$. 
 
 
For the purpose of this paper, we will consider the linearly ordered residuated lattices defined on real numbers from an interval $[a, b]$ ($a<b$), so $L$ will denote either a finite subset of $[a, b]$ such that $a$ and $b$ belong to $L$, or the whole interval. If $a$ and $b$ are not specified, the interval $[0,1]$ is assumed.  For simplicity, we will not mention that a residuated lattice is linearly ordered.

 


 \begin{ex}\label{ex1} Let us present some well-known adjoint pairs of multiplications and residua:
 \begin{itemize}
 \item God\"el system: $ \alpha \otimes_{G} \beta = \min(\alpha, \beta)$; 
  $ \alpha \to_{G} \beta = 1 \mbox{ if } \alpha \leq \beta \mbox{ and }  \beta 
  \mbox{ otherwise}.$ 
 \end{itemize}
For the following two systems, let $L = [0, 1]$.
\begin{itemize}
\item {\L}ukasiewicz system: $\alpha \otimes_L \beta = \max( 0, \alpha + \beta -1)$; 
 $\alpha \to_L \beta = \min( 1-\alpha+ \beta,1)$.
 \item Product system: $ \alpha \otimes_p \beta = \alpha \cdot\beta$; 
 $ \alpha \to_p \beta = \min( 1, \frac{\beta}{\alpha}) $ if $ \alpha \neq 0 $ and $1$ otherwise.
 \end{itemize}
 We have 
 $\neg_{L} \alpha= 1-\alpha $, 
 $\neg_{G} \alpha = \neg_p \alpha = 1$ if $\alpha=0$ and $0$ otherwise.
 \end{ex}
We say that 
that $L$ \emph{has no zero-divisors wrt. $\otimes$ } if it holds that if $\alpha\otimes \beta =0$ for $\alpha,\beta\in L$, then $\alpha=0$ or $\beta=0$. It is known that if $L$ has no zero divisors wrt. $\otimes$, then $\neg$ is \emph{strict}, i.e., $\neg \alpha =1$, if $\alpha=0$, and $\neg \alpha =0$, otherwise. Indeed, assume that $\alpha>0$ and $\neg\alpha>0$, then $\alpha \otimes \neg \alpha \le 0$ as a consequence of the adjointness property. But this is impossible, since $\alpha$ and $\neg\alpha$ would be zero divisors. 



\begin{ex}\label{ex2}
$(L,\wedge,\vee, \otimes_{G}, \to_G, 0, 1)$ is a finite or infinite commutative residuated lattice without zero divisors, where $\wedge$ and $\otimes_G$ coincide;\\
$([0, 1],\wedge,\vee, \otimes_{L}, \to_{L}, 0, 1)$ is an infinite residuated lattice with zero divisors;\\
$([0, 1] ,\wedge,\vee, \otimes_{p}, \to_p, 0, 1)$ is an infinite residuated lattice without zero divisors;
\end{ex}

\subsection{Qualitative integrals for input vectors with known values} \label{subsec.qi}

Let $L$ be a residuated lattice, let $\Univ =\{1, \cdots,n \}$ be a finite non-empty universe, and let $f : \Univ \to L$ denote an input vector. 
The qualitative integrals are aggregation operators defined using a capacity  (or fuzzy measure) $\mu: \mathcal{P}(\Univ) \to L$ (where $\mathcal{P}(\Univ)$ denotes the power set of $\Univ$), which is a monotonically non-decreasing set function such that $\mu(\emptyset)=0$ and $\mu(\Univ)=1$. For the purpose of this paper we use the qualitative integrals 
based on $\otimes $ or $\to$ introduced by Dubois, Prade and Rico in  \cite{DPR16} as follows:
\begin{align*}
\int_{\dpr,\mu}^{\otimes} f&= \bigvee_{A \subseteq \Univ} \big(\mu(A) \otimes \bigwedge_{i \in A} f_i \big)\\
\int_{\dpr,\mu}^{\to} f & = \bigwedge_{A\subseteq \Univ}  \big( \muc(A) \to \bigvee_{i \in A} f_i\big),
\end{align*}
where $\muc(A)=\neg \mu(\Univ\setminus A)$ for any $A\subseteq\Univ$.
It should be noted that these integrals were introduce to generalize the Sugeno integral in the framework of Kleene-Dienes system  $(\otimes_{KD},\to_{KD})$, where $\otimes_{KD}=\otimes_{G}$, but  $\to_{KD}$ is given by $\alpha \to_{KD} \beta = \max( 1-\alpha, \beta)$. Moreover,  similar qualitative integrals were introduced by Dvo\v{r}\'{a}k and Hol\v{c}apek in \cite{DH12}. A comparison of both types of qualitative integral can be found in \cite{HR20}. 


\section{Qualitative integrals on Dragonfly algebras}\label{s.int_on_DA}
\subsection{Dragonfly algebras}
Dragonfly algebras were introduced in \cite{SCBBD19} to define an appropriate structure to process data with unknown or missing values. The main idea of Dragonfly algebra is to extend the residuated lattice by one extra value denoted as $\star$ that represents an unknown value.  The extension is proposed to keep as many structural properties of residuated lattices as possible.

Let $L$ be a linearly ordered residuated lattice, and $L^* = L \cup \{*\}$. Denote $L^+ = L \setminus \{0\}$ and $\tilde{L} = L \setminus \{0, 1\}$. A Dragonfly algebra (determined by $L$) is an algebraic structure $L^*=(L^*, \wedge_D,\vee_D,\otimes_D,\rightarrow_D, 0, 1)$  with the operations defined in Tab.~\ref{tab.dfo}, where $\alpha, \beta \in \tilde{L}$.
\begin{table}[ht]
\centering{\begin{tabular}{c|c|c|c|c}
$\otimes_D$ & $\beta$ & $*$ & $0$ & $1$ \\
\hline 
$\alpha$ & $\alpha \otimes \beta$ & $*$ & $0$ & $\alpha$ \\
\hline 
$*$ & $*$ & $*$ & $0$ & $*$ \\
\hline 
$0$ & $0$ & $0$ & $0$ & $0$ \\
\hline 
$1$ & $\beta$ & $*$ & $0$ & $1$ 
\end{tabular}\quad
\begin{tabular}{c|c|c|c|c}
$\vee_D$ & $\beta$ & $*$ & $0$ & $1$ \\
\hline 
$\alpha$ & $\alpha \vee \beta$ & $\alpha$ & $\alpha$ & $1$ \\
\hline 
$*$ & $\beta$ & $*$ & $*$ & $1$ \\
\hline 
$0$ & $\beta$ & $*$ & $0$ & $1$ \\
\hline 
$1$ & $1$ & $1$&$1$ & $1$  
\end{tabular}}
\vspace{0.3cm}

\centering{\begin{tabular}{c|c|c|c|c}
$\wedge_D$ & $\beta$ & $*$ & $0$ & $1$ \\
\hline 
$\alpha$ & $\alpha \wedge \beta$ & $*$ & $0$ & $\alpha$ \\
\hline 
$*$ & $*$ & $*$ & $0$ & $*$ \\
\hline 
$0$ & $0$ & $0$ & $0$ & $0$ \\
\hline 
$1$ & $\beta$ & $*$ & $0$ & $1$ 
\end{tabular}\quad
\begin{tabular}{c|c|c|c|c}
$\to_D$ & $\beta$ & $*$ & $0$ & $1$ \\
\hline 
$\alpha$ & $\alpha \to \beta$ & $*$ & $\neg \alpha$ & $1$ \\
\hline 
$*$ & $\beta$ & $1$ & $*$ & $1$ \\
\hline 
$0$ & $1$ & $1$ & $1$ & $1$ \\
\hline 
$1$ & $\beta$ & $*$ &$0$ & $1$  
\end{tabular}\\[3pt]
}
\caption{Dragonfly algebra operations.}\label{tab.dfo}
\vspace{-0.5cm}
\end{table}
\noindent

Let us define $\neg_{\scriptscriptstyle{D}} \alpha = \alpha \to_D 0$. We should note that there are two orderings considered on a Dragonfly algebra: the first ordering $\leq$ is motivated by the \emph{lower estimation principle} used in designing the algebra and is defined as follows: $\alpha \leq \beta$ in $L^*$ iff $\alpha \leq \beta$ in $L$, or $\alpha = 0$ and $\beta = *$, or $\alpha = *$ and $\beta = 1$. Note that $*$ is not comparable to any $\alpha \in \tilde{L}$, hence, $\le$ is a partial ordering even if $L$ is linearly ordered. The second ordering $\leq_\ell$ is induced by one of the operations $\vee_D$, $\wedge_D$, i.e., $\alpha\le_\ell \beta$ if $\alpha\wedge_D\beta=\alpha$ or equivalently $\alpha\vee_D\beta=\beta$. Hence, one can simply find that $0 <_\ell * <_\ell \alpha$ for all $\alpha \in L^+$, and $\le_\ell$ is a linear ordering on $L^*$. 
The adjointness property~(\ref{eq.adjointness}) is not fulfilled neither for the ordering $\leq$, nor for $\leq_\ell$. For example, we have $\star\otimes_D \star \not\le_\ell 0$, but $\star\le_\ell \star\to_D 0$, or $\star\otimes_D 1 \le_\ell \beta$ for $\star<_\ell\beta<_\ell 1$, but $1\not\le_\ell \star\to_D \beta=\beta$. Hence, $L^*$ is not a residuated lattice (see \cite{SCBBD19} for details).  By Lemma~10 in \cite{SCBBD19}, the multiplication $\otimes_D$ is still monotonically increasing in both arguments, but the residuum $\to_D$ does not meet monotonic properties. Indeed, let $\alpha = * = \beta$ and $\gamma \in \tilde{L}$, i.e., $\gamma \in  L \setminus \{0, 1\}$. Then, $* = \beta \le_\ell \gamma$, but $1 = * \to_D * = \alpha \to_D \beta \not\leq_\ell \alpha \to_D \gamma = * \to_D \gamma = \gamma$. 

Obviously, $L$ can be embedded to $L^*$, but in contrast to the residuated lattice $L$, the Dragonfly algebra $L^*$ is a weaker algebraic structure without additional assumptions on~$L$. For example, $(L^*,\wedge_D,\vee_D)$ is a lattice in our case when $L$ is linearly ordered, but $(L^*,\otimes_D,1)$ is not a monoid even in this case. Indeed, if $\alpha\otimes \beta=0$ holds for $\alpha,\beta\in L$ such that $\alpha\neq 0\neq \beta$, that is  $\alpha$ and $\beta$ are \emph{zero divisors wrt.~$\otimes$}, then $\otimes_D$ is not associative, i.e.,  $(\alpha\otimes_D \beta)\otimes_D *=0\otimes_D *=0\neq *=\alpha\otimes_D *=\alpha\otimes_D (\beta\otimes_D *)$.  The following result is a consequence of a proposition proved in \cite{SCBBD19}.

\begin{proposition}
 Let $L^*$ be a Dragonfly algebra. If the underlying residuated lattice $L$ has no zero divisors wrt. $\otimes$, then $(L^*,\odot_D,1)$ is a monoid. 
\end{proposition}



For the Dragonfly algebra $L^*$, we have $\alpha\otimes \beta\neq *$ for any $\alpha,\beta\in L$, because $\alpha\otimes \beta\in L$ by the definition, and $\bigwedge_{D,i} a_i>_\ell*$ for any  sequence of $a_i \in L^+$, where $\bigwedge_D$ (and similarly $\bigvee_D$) denotes the infimum (supremum) in the Dragonfly algebra with respect to the lattice ordering. 


\subsection{Qualitative integrals on Dragonfly algebras}
Let $L^*$ be a Dragonfly algebra with the underlying residuated lattice that has no zero divisors wrt. $\otimes$. The qualitative integrals for input vectors with unknown values will be defined as a straightforward extensions of qualitative integrals introduced in Subsection~\ref{subsec.qi} where the values of input vectors belong to $L^*$. For our analysis, we also admit a capacity assigning the unknown value to a subset $A\subset \Univ$.  
Precisely, let $\Univ =\{1, \cdots,n \}$ be a finite non-empty universe,  and let $f : \Univ \to L^*$ be an input vector admitting the unknown value $\star$. 
Let $\mu:\mathcal{P}(\Univ) \to L^* $ be a fuzzy measure. The qualitative integrals based on $\otimes_D$ or $\to_D$ are defined as follows:
\begin{align*}
\int_{\dpr,\mu}^{\otimes_D} f&= \bigvee_{D,A \subseteq \Univ} \big(\mu(A) \otimes_D\bigwedge_{D,i \in A} f_i\big),\\
\int_{\dpr,\mu}^{\to_D} f & = \bigwedge_{D,A\subseteq \Univ}  \big( \muc(A) \to_D \bigvee_{D,i \in A} f_i\big),
\end{align*}
where $\muc(A) =\mu(\Univ\setminus A)\to_D 0$. 

Expressing an unknown value in the Dragonfly algebra, if $f$ is an input vector with unknown values, then $f$ provides a lower estimation of any input vector $g$ obtained from $f$ by replacing the unknown values by arbitrary non-zero values of $L$. 
Therefore, $f$ represents a sort of pessimistic attitude, and the qualitative integrals of $f$ are lower estimations of the qualitative integrals for all possible input vectors $g$ with known non-zero values, which is a consequence of the following theorems. 
\begin{theorem}\label{th.mono_integ}
Let $f\le_\ell g$ be two input vectors. Then
\begin{align*}
    \int_{\dpr,\mu}^{\otimes_D} f\le_\ell \int_{\dpr,\mu}^{\otimes_D} g 
\end{align*}
\end{theorem}
\begin{proof}
The  inequality immediately follows from $\alpha\otimes_D \beta\le_\ell \alpha\otimes_D \gamma$ that holds for any $\alpha,\beta,\gamma\in L^*$ such that $\beta\le_\ell \gamma$ (see Lemma~10 in \cite{SCBBD19}). 
\end{proof}

The monotonicity property is not true for $\int_{\dpr, \mu} ^ {\to_D} $ if the capacity is the unknown value for a certain subset, as the following example shows.

 
\begin{ex}
Let $L^*$ be a Dragonfly algebra such that there exists $\alpha \in L^*$ with $\star<_\ell \alpha<_\ell 1$. Consider $\Univ=\{1,2\}$ and the capacity $\mu$ defined as   
$\mu(\emptyset)=0$, $\mu(\{1\})=\mu(\{2\})=\star$,  $\mu(\{1,2\})=1$. Let $f, g:\Univ\to L^*$ be given by 
$f_1=\star$, $f_2=1$, $g_1=\alpha$ for $\star<\alpha<1$ and $g_2=1$. Obviously, $f<_\ell g$. It is easy to see that  $\mu^c(A)=\neg_{\scriptscriptstyle{D}} \mu(\Univ\setminus A)=\mu(A)$ for any  $A\subseteq \Univ$. Then
\begin{gather*}
    \int_{\dpr,\mu}^{\to_D} f = (\star\to_D \star) \wedge_D (\star\to_D 1) \wedge_D (1\to_D 1) =1 >_\ell\\ (\star\to_D \alpha) \wedge_D (\star\to_D 1) \wedge_D (1\to_D 1) = \alpha = \int_{\dpr,\mu}^{\to_D} g, 
\end{gather*}
which violates the monotonicity of the residuum based qualitative integral. 
\end{ex}

\begin{theorem}
Let $f\le_\ell g$ be two input vectors, and  the capacity $\mu$ takes values only in $L$. Then 
 \begin{gather*}
    \int_{\dpr,\mu}^{\to_D} f\le_\ell \int_{\dpr,\mu}^{\to_D} g.
\end{gather*}
 \end{theorem}
 
 \begin{proof}
The inequality is a straightforward consequence of $\alpha \to_D \beta \leq \alpha \to_D \gamma $ for any $\alpha \in L$ and  $\beta,\gamma\in L^*$ such  $\beta \leq \gamma$, which holds in Dragonfly algebras. Indeed, the inequality trivially holds for $\alpha\in \{0,1\}$ and $\gamma=1$. If $\alpha\not\in \{0,1\}$, then it is sufficient to verify a) $\beta=0<_\ell * =\gamma$ and b) $\beta=*<_\ell \gamma<_\ell 1$, the rest follows from the monotonicity of $\to_D$ on the underlying residuated lattice $L$. For a), we have $\alpha\to_D 0=0<_\ell *= \alpha \to_D *$, and for b), we have $\alpha\to *=*<_\ell \gamma\le_\ell \alpha\to_D \gamma$.
 \end{proof}
 
 Note that the above proposition works for capacities admitting unknown values and three-valued $L^*$, i.e., $L^* = \{0, *, 1\}$, since then the property in concern holds. Furthermore, the previous theorems are not true without the additional assumption saying that the underlying residuated lattice has no zero divisors.

 
 
 
 
 

An interesting question is, when the qualitative integrals return a known value and when an unknown value. The following theorems specify the cases in which the multiplication-based qualitative integrals return a known value. 

\begin{theorem}
\label{intdpr*}
$\int_{\dpr,\mu}^{\otimes_D} f=0$ if and only if $\mu(\{i\mid f_i>_\ell0\})=0$.
\end{theorem}

\begin{proof}
($\Rightarrow$) 
Let $\int_{\dpr,\mu}^{\otimes_D} f=0$ and assume that $\mu(\{i\mid f_i>_\ell 0\})>_\ell 0$. Put $X=\{i\mid f_i{>_\ell}0\}$. Then $\bigwedge_{D,i \in X} f_i\geq_\ell *$ and $\mu(X)\geq_\ell *$, which implies $\int_{\dpr,\mu}^{\otimes_D} f\geq_\ell *>_\ell0$, a contradiction. 

($\Leftarrow$) If $\mu(\{i\mid f_i>_\ell 0\})=0$, then $\mu(A)=0$ is true for any  $A\subseteq \{i\mid f_i>_\ell 0\}$. Hence, even if $\bigwedge_{D,i \in A} f_i>_\ell0$ for a certain subset $A$, we find that $\mu(A)\otimes\bigwedge_{D,i \in A} f_i=0$. As a consequence of this we find that $\int_{\dpr,\mu}^{\otimes_D} f=0$. 
\end{proof}

\begin{theorem}\label{intdpr**}
$\int_{\dpr,\mu}^{\otimes_D} f>_\ell *$ if and only if $\mu(\{i\mid f_i>_\ell *\})>*$.
\end{theorem}
\begin{proof}
($\Rightarrow$) If $\int_{\dpr,\mu}^{\otimes_D} f>_\ell *$, then there exists $A$ such that $\mu(A)\otimes \bigwedge_{D,i \in A} f_i >_\ell*$, which implies $\mu(A)>_\ell*$. Hence, we obtain that $\mu(\{i\mid f_i>_\ell *\})>_\ell*$ as a consequence of the monotonicity of $\mu$ and the fact that $A\subseteq \{i\mid f_i>_\ell *\}$. 

($\Leftarrow$) If $\mu(\{i\mid f_i>_\ell *\})>_\ell*$, then  $\bigwedge_{D,i \in \{i\mid f_i>_\ell *\}} f_i >_\ell*$, which implies $ \mu(\{i\mid f_i>_\ell *\})\otimes \bigwedge_{D,i \in \{i\mid f_i> *\}} f_i >_\ell*,$
where we used the assumption that the underlying residuated lattice has no zero divisors.
\end{proof}

To summarize, the multiplication-based qualitative integrals  return a value in $L$ if and only if there is at least one value known for the input vector $f$ and if the capacity of the set of criteria with known evaluations is known. 

\begin{ex}
If the capacity $\mu$ is ``absolutely'' unknown, i.e., $\mu(A) = * $, if $\emptyset\neq A \neq \Univ $, $\mu(\emptyset)= 0$ and $\mu(\Univ)= 1$, then 
$\int_{\dpr,\mu}^{\otimes_D} f= \bigwedge_{D,i \in \Univ} f_i$ for any input vector with known values such that $f_i=0$ for all $i\in \Univ$ or $f_i>0$ for all $i\in \Univ$. If there are $f_i=0$ and $f_j>0$ for certain $i,j\in \mathcal{C}$, then  $\int_{\dpr,\mu}^{\otimes_D} f =\star$ as a consequence of Theorem~\ref{intdpr*} (see Corollary~\ref{cor.intdprspec*} below). 
\end{ex}

\begin{ex}
We consider $\Univ=\{1,2,3\}$, $f=(f_1,f_2,f_3)=(*,*,\alpha)$\, where $\alpha \in L^+$.
\begin{itemize}
    \item Let us suppose that
$\mu(A) =* $ if $\{3\} \subseteq A$, $A \neq \Univ$, $\mu(\Univ)=1$ and $\mu(A) =0$ otherwise. \\
 We have $\int_{\dpr,\mu}^{\otimes_D} f= \vee_D(*, \alpha \otimes_D *, * \otimes_D * )= *.$
 \item Let us suppose $\mu(A) = \alpha $ if $\{3\} \subseteq A$, $A \neq \Univ$, $\mu(\Univ)=1$ and $\mu(A) =0$ otherwise.\\ 
 We have $\int_{\dpr,\mu}^{\otimes_D} f= \vee_D(*, \alpha \otimes_D \alpha, * \otimes_D \alpha )= \alpha\otimes_D \alpha.$ 
 \end{itemize}
\end{ex}

Using Theorems~\ref{intdpr*} and \ref{intdpr**}, we can specify the cases when the multiplication-based integrals return the unknown value.  
\begin{theorem}\label{thm.intdpr*}
$\int_{\dpr,\mu}^{\otimes_D} f = *$ if and only if $\mu(\{i\mid f_i>_\ell *\}) \le_\ell *$, $\mu(\{i\mid f_i\ge_\ell *\}) \ge_\ell*$ and   $f_i>_\ell 0$ for at least one $i\in \Univ$. 
\end{theorem}

\begin{proof}
($\Rightarrow$) By Theorem~\ref{intdpr*}, we find that  $\mu(\{i\mid f_i>0\}) \not= 0$ and simultaneously $\mu(\{i\mid f_i> *\}) \not>_\ell *$. Since $\mu(\emptyset)=0$, we immediately obtain that there exists $f_i>_\ell0$ for certain $i\in \Univ$. 
Moreover, we simply have $\mu(\{i\mid f_i>_\ell 0\})=\mu(\{i\mid f_i\ge_\ell *\}) \ge_\ell *>_\ell0$ and $\mu(\{i\mid f_i>_\ell *\}) \le_\ell *$. 

($\Leftarrow$) Assume that $\int_{\dpr,\mu}^{\otimes_D} f \neq *$.  Then $\int_{\dpr,\mu}^{\otimes_D} f >_\ell *$ or 
$\int_{\dpr,\mu}^{\otimes_D} f = 0$. By Theorem~\ref{intdpr*},  we find that
 $\mu(\{i\mid f_i>_\ell * \})>_\ell *$ or $\mu(\{i\mid f_i>_\ell 0\}) = \mu(\{i\mid f_i \geq_\ell *\})=0$
which is a contradiction, therefore, $\int_{\dpr,\mu}^{\otimes_D} f = *$.
\end{proof}

The following corollary characterizes when the multiplication-based qualitative integral provide the unknown value as its result if an input vector is such that all its values are known. 
\begin{corollary}\label{cor.intdprspec*} Let $f$ be an input vector such that $f_i\in L$ for any $i=1,\dots, n$. Then
$\int_{\dpr,\mu}^{\otimes_D} f = *$ if and only if $\mu(\{i\mid f_i>_\ell *\}) = *$ and   $f_i>_\ell 0$ for at least one $i\in \Univ$. 
\end{corollary}
\begin{proof}
It immediately follows from Theorem~\ref{thm.intdpr*} and the fact that $\{i\mid f_i>_\ell *\}=\{i\mid f_i\ge_\ell *\}$.
\end{proof}

Obviously, if the capacity is known everywhere and is non-zero for the set of all criteria which have known and non-zero evaluation, then the multiplication-based qualitative integrals always give a global evaluation in $L^+$. Indeed, by the assumption on $\mu$ we have $\mu(\{i\mid f_i>_\ell\star\})>_\ell *$, which implies $\int_{\dpr,\mu}^{\otimes_D} f>_\ell *$ due to Theorem~\ref{intdpr**}. 

The next two theorems specify the cases in which the residuum-based qualitative integrals return a known value. 


\begin{theorem}\label{thm.5}
$\int_{\dpr,\mu}^{\to_D} f= 0$  if and only if $\mu^c(\{i\mid f_i=0\})=1$.
\end{theorem}
\begin{proof}
($\Rightarrow)$ If $\int_{\dpr,\mu}^{\to_D} f= 0$, then  there exists $A$ such that $\mu^c(A)\to_D \bigvee_{D,i\in A} f_i=0$. We show that the equality is satisfied if and only if $\mu^c(A)=1$ and $\bigvee_{D,i\in A} f_i=0$. First, we should note that the negation is strong because of our assumption on the underlying residuated lattice that has no-zero divisors, therefore, $\mu^c(A)\in \{0,\star, 1\}$ for any $A\subseteq \Univ$, where $\neg_{\scriptscriptstyle{D}} \star =\star$. If $\mu^c(A)= 0$, then trivially $\mu^c(A)\to \bigvee_{D,i\in A} f_i=1$, which is a contradiction, therefore, $\mu^c(A)>_\ell 0$. If $\mu^c(A)=\star$, then  $\mu^c(A)\to_D \bigvee_{D,i\in A} f_i\geq_\ell \star$, which is again a contradiction, therefore, $\mu^c(A)= 1$. Hence, we obtain  $\bigvee_{D,i\in A} f_i=1\to_D \bigvee_{D,i\in A} f_i= 0$, and thus $\bigvee_{D,i\in A} f_i=0$. From the monotonicity of the capacity $\mu$, we obtain $\mu^c(\{i\mid f_i=0\})=1$. 

($\Leftarrow$) Put $A=\{i\mid f_i=0\}$. If $\mu^c(A)=1$, then $\mu^c(A)\to_D \bigvee_{D,i\in A} f_i=1\to_D 0=0$, therefore, $\int_{\dpr,\mu}^{\to_D} f= 0$. 
\end{proof}

\begin{theorem}\label{thm.6}
$\int_{\dpr,\mu}^{\to_D} f>_\ell \star $  if and only if $\mu^c(\{i\mid f_i=0\})=0$ and $\mu^c(\{i\mid f_i\le_\ell \star\})<_\ell1$.
\end{theorem}
\begin{proof}
($\Rightarrow)$ If $\int_{\dpr,\mu}^{\to_D} f>_\ell \star$, then for any $A\subseteq \Univ$, we have $\mu^c(A)\to_D \bigvee_{D,i\in A} f_i>_\ell \star$. Recall that $\mu^c(A)\in \{0,\star,1\}$ for any $A\subseteq \Univ$, and $\neg_{\scriptscriptstyle{D}} \star =\star$. Assume that $\mu^c(\{i\mid f_i=0\})>_\ell 0$. Then  $\mu^c(\{i\mid f_i=0\})\to_D \bigvee_{D,i\in \{i\mid f_i=0\}} f_i=\neg_{\scriptscriptstyle{D}} \mu^c(\{i\mid f_i=0\})\le_\ell\star$, which is a contradiction. Hence, we obtain $\mu^c(\{i\mid f_i=0\})=0$.  Further, assume that $\mu^c(\{i\mid f_i\le_\ell\star\})=1$. From $\mu^c(\{i\mid f_i=0\})=0$, there is  $f_i=\star$ for at least one $i\in \Univ$. Put $A=\{i\mid f_i\le_\ell \star\}$. Then $\mu^c(A)\to_D \bigvee_{D,i\in A} f_i=1\to_D \star=\star$, which is a contradiction. Hence, we obtain that $\mu^c(\{i\mid f_i\le_\ell \star\})<1$. 

($\Leftarrow$) Since $\mu^c(\Univ)=1$, from the assumption, we find that there is  $f_i>_\ell \star$ for at least one $i\in \Univ$. Let $A\subseteq \Univ$. If $f_i=0$ for any $i\in A$, from the monotonicity of $\mu^c$ and the assumption $\mu^c(\{i\mid f_i=0\})=0$, we obtain $\mu^c(A)=0$. Hence, we get $\mu^c(A)\to_D \bigvee_{D,i\in A} f_i=0\to_D 0=1$, and $A$ can be ignored in the calculation of the qualitative integral. Further, if $f_i\le_\ell \star$ for any $i\in A$ and there is $f_i=\star$ for a certain $i\in A$, we obtain $\mu^c(A)\le_\ell \star$. Hence, we get $\mu^c(A)\to_D \bigvee_{D,i\in A} f_i=\mu^c(A)\to_D \star=1$, and again $A$ can be ignored in the calculation of the qualitative integral. We denote $\mathcal{D}=\{A\subseteq \Univ \mid f_i>_\ell \star \hbox{ for at least one } i\in A\}$. Obviously, $\mathcal{D}\neq \emptyset$. By the previous analysis, we find that
\begin{align*}
    \int_{\dpr,\mu}^{\to_D} f =\bigwedge_{D, A\in \mathcal{D}} \mu^c(A)\to_D \bigvee_{D, i\in A} f_i>_\ell \star,
\end{align*}
where we used $\mu^c(A)\to_D \bigvee_{D, i\in A} f_i\in \{\bigvee_{D, i\in A} f_i,1\}$.
\end{proof}

Now, we can advance to a characterization of cases when the residuum-based integrals return the unknown value.

\begin{theorem}\label{thm.7}
$\int_{\dpr,\mu}^{\to_D} f= *$  if and only if there exists  $f_i\in \{0,*\}$ for a certain $i\in \Univ$ and one of the following conditions is satisfied:
\begin{enumerate}[a)]
    \item $\mu^c(\{i\mid f_i=0\})=*$,
    \item $\mu^c(\{i\mid f_i=0\})=0$, if $f_i=0$ for at least one $i\in \Univ$, and $\mu^c(\{j\mid f_j\le_\ell *\})=1$, 
    \item $\mu^c(\{i\mid f_i= *\})=1$, if $f_i>_\ell 0$ for any $i\in\Univ$. 
\end{enumerate}  
\end{theorem}

\begin{proof}
($\Rightarrow$) Obviously, if  $\int_{\dpr,\mu}^{\to_D} f= *$, then by Theorems~\ref{thm.5}~and~\ref{thm.6}, we simply obtain,  as the negations of the respective statements, that 
$\mu^c(\{i\mid f_i=0\}) <_\ell1$ and $(\mu^c(\{i\mid f_i=0\}) >_\ell 0$ or $\mu^c(\{i\mid f_i\le_\ell *\}) = 1$). First, let us consider that there is $f_i=0$ for at least one $i\in \Univ$. We obtain that $\mu^c(\{i\mid f_i=0\}) <_\ell 1$ and $(\mu^c(\{i\mid f_i=0\}) >_\ell 0$, which  results to $\mu^c(\{i\mid f_i=0\})=*$ stated in a), or $\mu^c(\{i\mid f_i=0\}) <_\ell 1$ and $\mu^c(\{j\mid f_j\le_\ell *\})=1$. Since $\mu^c(\{i\mid f_i=0\})=*$ is already included in a), we can omit it here. Moreover, $\mu^c(\{j\mid f_j\le_\ell *\})=1$ can be true only if there is $f_j=*$ for at least one $j\in \Univ$. 
By combining these conditions we obtain b). Further, let us assume that $f_i\neq 0$ for any $i\in \Univ$. It is easy to see that  $\mu^c(\{i\mid f_i=0\}) <_\ell1$ and $(\mu^c(\{i\mid f_i=0\}) >_\ell 0$ cannot occur, because $\mu^c(\{i\mid f_i=0\}) = 0$. So we just have  $\mu^c(\{i\mid f_i=0\}) <_\ell 1$ and $\mu^c(\{j\mid f_j\le_\ell *\})=1$. Since $f_i\neq 0$ for any $i\in \Univ$, we find that $\mu^c(\{i\mid f_i=0\}) <_\ell 1$ is trivially true and it must hold $\mu^c(\{j\mid f_j= *\})=1$, where we assume that there exists at least one $f_j=*$ for $j\in \Univ$. 
Combining these conditions we obtain c). To complete this part of the proof, we have to realize that a)-c) cannot occur if $f_i\not\in \{0,\star\}$ for all $i\in \Univ$.


($\Leftarrow$) Assume that $f_i\in \{0,\star\}$ for a certain $i\in \Univ$. First, let  a) holds. Obviously, there is $f_i=0$ for at least $i\in \Univ$, and  we have $\mu^c(\{i\mid f_i=0\})\to_D \bigvee_{D,i\in \{i\mid f_i=0\}} f_i=\star\to_D 0=\star$. Assume that there is $A\subseteq \Univ$ such that $\mu^c(A)\to_D\bigvee_{D, i\in A} f_i=0$. Since $\mu^c(A)\in \{0,\star,1\}$, we simply find that $\mu^c(A)=1$ and $\bigvee_{D, i\in A} f_i=0$ to obtain the residuum equal to $0$. But $A\subseteq \{i\mid f_i=0\}$ and thus $\mu^c(A)\le_\ell \mu^c(\{i\mid f_i=0\})= \star$, which is a contradiction. Hence, there is no subset $A\subseteq \Univ$ such that $\mu^c(A)\to_D\bigvee_{D, i\in A} f_i=0$, and we get $\int_{\dpr,\mu}^{\to_D} f= *$.  

Further, let  b) holds.  Since there is $f_i=\star$ for a certain $i\in \Univ$ (otherwise, $\mu^c(\{j\mid f_j\le_\ell *\})\not=1$), we obtain $\mu^c(\{i\mid f_i\le_\ell \star\})\to_D \bigvee_{D, i\in \{i\mid f_i\le_\ell \star\}} f_i=1\to_D \star=\star$. Assume that there is $A\subseteq \Univ$ such that $\mu^c(A)\to_D\bigvee_{D, i\in A} f_i=0$. Similarly to  case a), the equality is satisfied only if $\mu^c(A)=1$ and $\bigvee_{D,i\in A}f_i=0$. But for any $A\subseteq \{i\mid f_i=0\}$, we have $\mu^c(A)\le_\ell \mu^c(\{i\mid f_i=0\})=0$, therefore, there is no subset $A$ such that $\mu^c(A)\to_D\bigvee_{D, i\in A} f_i=0$, and we get $\int_{\dpr,\mu}^{\to_D} f= *$. 

Finally, let c) holds. Similarly to case b), we have $\mu^c(\{i\mid f_i=\star\})\to_D \bigvee_{D, i\in \{i\mid f_i= \star\}} f_i=1\to_D \star=\star$. Assume that there is $A\subseteq \Univ$ such that $\mu^c(A)\to_D\bigvee_{D, i\in A} f_i=0$. Again, the equality holds only if $\mu^c(A)=1$ and $\bigvee_{D,i\in A} f_i=0$. Since $f_i\neq 0$ for any $i\in \Univ$, we find that $A=\emptyset$, but $\mu^c(\emptyset)=0$, therefore, there is no subset $A$ such that  $\mu^c(A)\to_D\bigvee_{D, i\in A} f_i=0$, and we get $\int_{\dpr,\mu}^{\to_D} f= *$, which concludes the proof.
\end{proof}

\begin{ex}
If the capacity $\mu$ is unknown, $\mu(A)= *$ everywhere excepts on $\emptyset$ and $\Univ$ and if all values of $f$ are known and different from $0$ then $\int_{\dpr,\mu}^{\to_D} f = \bigwedge_{D,i \in \Univ} f_i$.
\end{ex}

\begin{ex}
We consider $\Univ=\{1,2,3\}$, $f=(f_1,f_2,f_3)=(*,*,\alpha)$, where $\alpha \in L^+$.
\begin{itemize}
    \item Let us suppose 
$\mu(A) =* $ if $\{3\} \subseteq A$, $A \neq \Univ$ $\mu(\Univ)=1$ and $\mu(A) =0$ otherwise. 
 We have $\int_{\dpr,\mu}^{\to_D} f= (* \to_D * )\wedge_D (1 \to_D \alpha) = \alpha.$
\item 
Let us suppose $\mu(A) = 1 $ if $\{1,2\} \subseteq A$ and $\mu(A)=0$ otherwise. We have $\int_{\dpr,\mu}^{\to_D} f= *$.

 \end{itemize}
\end{ex}





\section{Knowledge representation model for data with unknown values}\label{s.example_unknown_values}
In this section, we demonstrate how to model expert knowledge if there are unknown values in data. We assume that the data is a collection of pairs $(f_j, \alpha_j)$, $j=1,\dots,m$, where $f_j$ is the input vector providing the evaluation of the $j$-th object according to the given criteria and $\alpha_j$ represents the global evaluation of this object given by an expert. The expert knowledge is then modeled with such capacities that the qualitative integrals return expert global evaluations to given input vectors (see, e.g., \cite{DDPRF15}). 
In practice, we face the problem that the data are incomplete, specifically some evaluations of the object are unknown. In the next subsections, we show that the expert knowledge can be successfully modeled if the unknown values are interpreted as $*$ in a Dragonfly algebra.

For the illustrative purpose, we  consider the data presented in \cite{F}
that were collected on retention basins in Lyon to model the impact of four indicators, namely, Esterase Activity ($\AE$), Alkaline Phosphatase Activity ($\APA$), Chlorophyll Fluorescence ($\F$), and Algal Growth ($\C$), on the water ecosystem health. Data are displayed in Table~\ref{Tab2} and consist of $14$ measurements of the indicators (i.e., $\Univ=\{\AE, \APA, \F, \C\}$) and  the expert global evaluations of the impact  on the water ecosystem health 
(denoted by $\alpha$). Note that there are two unknown values indicated by the empty place. We use the same evaluation scale as in \cite{DDPRF15}, i.e., the totally ordered scale $L =\{1,2,3,4,5\}$, and the Dragonfly algebra $L^*$ with the underlying residuated lattice $L$ defined by the G\"{o}del system (see Examples~\ref{ex1}~and~\ref{ex2}). We restrict our analysis to the multiplication-based qualitative integral $\int_{\dpr,\mu}^{\wedge}$, which coincides with the Sugeno integral. 
\begin{table}[ht]
\begin{center}  
{\small
$
\begin{array}{ll}
\begin{array}{|c|c|c|c|c|}
\hline
\AE & \APA & \F& \C  & \alpha\\
\hline
4    &  2   & 3  &  3  & 3  \\
4    &   2  & 1  & 1  & 2  \\
2  &         & 3  & 1  & 2 \\
2  & 4      & 2  & 1  & 2 \\
5 & 4      &  3 & 1 & 3 \\
3  &         & 4  & 3 & 3 \\
1  &  3     & 5  & 4 & 3 \\
\hline
\end{array}
&
\begin{array}{|c|c|c|c|c|}
\hline
\AE & \APA & \F& \C  & \alpha \\
\hline
1 &   5 & 3 & 3 & 3 \\
1  &  1   & 4 & 2 & 2 \\
2  & 3    & 3 & 3 & 3\\
5  & 2   & 2 & 1 & 2 \\
4  & 5   & 4 & 2 & 4\\
3 & 4 & 3 & 1 & 3 \\
1   &  1 & 5 & 5 & 3 \\
\hline
\end{array}
\end{array}
$
}
\end{center}
\caption{Evaluation of criteria with unknown values.}\label{Tab2}
\vspace{-0.5cm}
\end{table}

To calculate the admissible capacities, we extend  the process presented  in \cite{RLGC02}  to Dragonfly algebras. More specifically, for any pair $(f,\alpha)$, where $\alpha\in L^*$, 
the set of admissible capacities is given as 
$M_{f,\alpha}=\{\mu\mid  \int_{\dpr,\mu}^{\wedge} f=\alpha \}=\{\mu\mid  \check{\mu}_{f,\alpha} \leq_\ell \mu \leq_\ell \hat{\mu}_{f,\alpha}\}$,
where \vspace{3pt}

\noindent $\bullet\,  \check{\mu}_{f,\alpha}(A) = 
\alpha$, if $\{i\mid f_i\geq_\ell \alpha \}\subseteq A$, $\check{\mu}_{f,\alpha}(A) = 
1$, otherwise,

\noindent $\bullet\,  \hat{\mu}_{f,\alpha}(A)=
\alpha$, if  $A \subseteq \{i\mid f_i>_\ell \alpha\}$, $\check{\mu}_{f,\alpha}(A) = 
5$, otherwise,

\vspace{6pt}

\noindent and $\check{\mu}(\emptyset)=\hat{\mu}(\emptyset)=1$, $\check{\mu}(\Univ)=\hat{\mu}(\Univ)=5$. A capacity $\mu$  is admissible with regard to data given as $(f_j,\alpha_j)$, $j=1,\dots,m$, provided that $\mu\in \bigcap_{j=1}^m M_{f_j,\alpha_j}$. It should be noted that, given data, the existence of an admissible capacity is not guaranteed.


\subsection{Expert knowledge model determined from complete data}
Let us assume that the admissible capacities for data in Table~\ref{Tab2} are calculated only from the input vectors whose all coordinates are known, in other words, the incomplete input vectors are ignored. The family of all admissible capacities expressed using bounds, i.e., the lower capacity $\check{\mu}$ and the upper capacity $\hat{\mu}$, is displayed in Table~\ref{Tab3} (cf. \cite{DDPRF15}).

\begin{table}[h]
$$
\begin{array}{|l|c|c||l|c|c||l|c|c|}
\hline
criteria              & \check{\mu} & \hat{\mu} &
criteria              & \check{\mu} & \hat{\mu} & 
criteria              & \check{\mu} & \hat{\mu}   \\
\hline
 \scriptstyle\{ \AE\}          &  1         & 2            & 
\scriptstyle\{ \APA\}          &  1          & 2          & 
\scriptstyle\{ \F\}          & 1           & 2           \\
\scriptstyle\{ \C\}           & 1            & 3          &
\scriptstyle\{ \AE, \APA\}        & 2         & 3           &  
\scriptstyle\{ \AE, \F\}        & 1         & 5           \\
\scriptstyle\{ \AE, \C\}        & 1         & 5            &   
\scriptstyle\{ \APA, \F\}        & 1         & 5          &    
\scriptstyle\{ \APA, \C\}        & 1         & 5            \\
\scriptstyle\{ \F,\C\}        & 3         & 3              &
\scriptstyle\{ \AE,\APA, \F\}        & 4         & 5        &    
\scriptstyle\{ \AE,\APA,\C\}        & 2         & 5          \\
\scriptstyle\{ \AE,\F,\C\}        & 3         & 5           &
\scriptstyle\{ \APA, \F,\C\}        & 3         & 5          &    
\scriptstyle \{ \AE, \APA, \F,\C\}    & 5         & 5  \\
\hline
\end{array}
$$
\caption{Family of admissible capacities for known input vectors.}
\vspace{-0.5cm}
\label{Tab3}
\end{table}

%



Now, let us express the unknown values by $*$ (a lower estimation different from the least evaluation specifying no impact on the water ecosystem health) and consider the calculation of the qualitative integral over the Dragonfly algebra $L^*$ for the capacity $\check{\mu}$. For example, 
we have $\int_{\dpr,\check{\mu}}^{\wedge_D}(2,*,3,1) =1 \neq 2$, where $2$ is the expert global evaluation. 
On the other side, 
$\int_{\dpr,\check{\mu}}^{\wedge_D}(3,*,4,3) =3 $, which perfectly fits the expert global evaluation. Therefore, the expert knowledge described by admissible capacities works unreliably for input vectors with unknown components, although we used the lower estimation strategy. 
But it is not surprising, since data $(f_3, 2)$ and $(f_6, 3)$ were not used for the determination of the family of admissible capacities in Table~\ref{Tab3}, and clearly $*$ effects the calculation. 
On the other hand, capacities calculated from complete input vectors can reveal possible values for unknown components in incomplete vectors that fit the expert knowledge.



\begin{ex}Consider $\check{\mu}$ from Table~\ref{Tab3} and the incomplete input vector  $f_3=(2,\, ,3,1)$. 
%
Define $ f_3 '= (2, \{2,3\}, 3,1) $ as a new input vector with the unknown value  expressed by the interval $\{2,3\}$ instead of $*$, which failed to provide the correct global evaluation. It is easy to show that  $\int_{\dpr,{\check{\mu}}}^\wedge (2,2,3,1) =\int_{\dpr,{\check{\mu}}}^\wedge(2,3,3,1)= 2$. So, the values $2$ and $3$ (i.e., the interval $\{2,3\}$) seem to be more appropriate representatives of the unknown value than the lower estimation given by $*$. This observation initiates our future research on the use of intervals on  lattices to express unknown values, while global evaluation is provided by qualitative intervals for input interval vectors.


\end{ex}
\subsection{Expert knowledge model determined from incomplete data}
Let us express the unknown values by $*$ from the Dragonfly algebra. The natural question is whether we are able to calculate the family of admissible capacities for such data directly considering the qualitative integral on the Dragonfly algebra. The answer is positive and the family of all admissible capacities expressed using bounds   $\check{\mu}$ and $\hat{\mu}$ is in Table~\ref{tab4}.


\begin{table}[h]
{\small
$$
\begin{array}{|l|c|c||l|c|c||l|c|c|}
\hline
criteria              & \check{\mu} & \hat{\mu} &
criteria              & \check{\mu} & \hat{\mu} & 
criteria              & \check{\mu} & \hat{\mu}   \\
\hline
 \scriptstyle\{ \AE\}          &  1         & 2            & 
\scriptstyle\{ \APA\}          &  1          & 2          & 
\scriptstyle\{ \F\}          & 1           & 2           \\
\scriptstyle\{ \C\}           & 1            & 3          &
\scriptstyle\{ \AE, \APA\}        & 2         & 3           &  
\scriptstyle\{ \AE, \F\}        & 2        & 5           \\
\scriptstyle\{ \AE, \C\}        & 1         & 5            &   
\scriptstyle\{ \APA, \F\}        & 1         & 5          &    
\scriptstyle\{ \APA, \C\}        & 1         & 5            \\
\scriptstyle\{ \F,\C\}        & 3         & 3              &
\scriptstyle\{ \AE,\APA, \F\}        & 4         & 5        &    
\scriptstyle\{ \AE,\APA,\C\}        & 2         & 5          \\
\scriptstyle\{ \AE,\F,\C\}        & 3         & 5           &
\scriptstyle\{ \APA, \F,\C\}        & 3         & 5          &    
\scriptstyle \{ \AE, \APA, \F,\C\}    & 5         & 5  \\
\hline
\end{array}
$$
}
\caption{Family of admissible capacities for all vectors.}\label{tab4}
\vspace{-0.5cm}
\end{table}

One can see that there is only one change compared to the family of capacities presented in Table~\ref{Tab3}, namely, we have  $\check{\mu}(\{A,C\})=2$ instead of $\check{\mu}(\{A,C\}) = 1$. This change ensures the correct global evaluations of all input vectors including $ (2, *, 3,1) $, so the use of Dragonfly algebra seems to be advantageous, but a deeper analysis is a subject of our future research. Note that the interval for admissible capacities for the set $\{A,C\}$ becomes narrower, which means that we obtained more precise knowledge of capacities that represent the data. 
Finally, it is interesting that adding  $(f,\alpha)= ((*,*,2,*),*)$ to data, we obtain the unknown value of the capacity $\hat{\mu }$ for $\{C\}$, i.e.,  $\hat{\mu }(\{C\})=*$, compared to the original value equal to $2$. 

\bibliographystyle{unsrt}  
\bibliography{biblio}  
\end{document}